\documentclass{amsart}

\theoremstyle{plain}

\newtheorem{proposition}{Proposition}

\newcommand{\inv}{^{-1}}
\newcommand{\diag}{\mathrm{diag}}

\newcommand{\fG}{\mathfrak{G}}
\newcommand{\fH}{\mathfrak{H}}
\newcommand{\bbF}{\mathbb{F}}
\newcommand{\bfM}{\mathbf{M}}

\newcommand{\wt}[1]{\widetilde{#1}}

\makeatletter
\newcommand{\cmpl}[1]{%
    \sbox\z@{$#1$}%
    \dimen@=\wd\z@
    \advance \dimen@ -\strip@pt\fontdimen\@ne\textfont\@ne \ht\z@
    \setbox\tw@=\hb@xt@\dimen@{}%
    \ht\tw@=\ht\z@ \dp\tw@=\dp\z@
    \box\z@
    \llap{$\overline{\box\tw@}$}%
} 
\makeatother

\begin{document}

\title{On the Counterexamples to the Unit Conjecture for Group Rings}

\author{D. S. Passman}
\address{Department of Mathematics, University of Wisconsin-Madison, Madison,
Wisconsin 53706, passman@math.wisc.edu}

\begin{abstract}
We offer two comments on the beautiful papers of Giles Gardam and Alan Murray
that yield counterexamples to the Kaplansky unit conjecture. First we discuss the determinants of these units
in a certain $4\times 4$ matrix representation of the group ring. Then we explain why
there is a doubly infinite family of units in the Murray paper.
\end{abstract}

\maketitle

Let $\mathfrak{G}=\langle a,b \mid (a^2)^b=a^{-2}, \,(b^2)^a= b^{-2} \rangle$. Then we know that
$\mathfrak{G}$ is a torsion-free group  with a normal abelian subgroup
$\fH$ of index 4 and with $\fG/\fH$ a fours group. The paper \cite{G} offers an example
of a nontrivial unit in the group algebra $\bbF_2[\fG]$ where $\bbF_2=\mathrm{GF}(2)$. Building on
that, \cite{M} offers a doubly infinite family of nontrivial units in $\bbF_d [\fG]$ for any prime $d$
where $\bbF_d=\mathrm{GF}(d)$. Of course a unit is nontrivial if it is not a scalar multiple of an
element of $\fG$.

Now it is a standard fact that $\bbF_d[\fG]$ embeds in the $4\times 4$ matrix ring over $\bbF_d[\fH]$.
Indeed write $\mathbb{V} =\bbF_d[\fG]$. Then $\mathbb{V}$ is a faithful right $\bbF_d[\fG]$-module
via right multiplication and $\mathbb{V}$ is a free left $\bbF_d[\fH]$-module via left multiplication where
the coset representatives $1,a,b,c=ab$ of $\fH$ in $\fG$ yield a free basis for $\mathbb{V}$.
Since right and left multiplication commute as operators on $\mathbb{V}$, it follows that
$\bbF_d[\fG]$ embeds in the $\bbF_d[\fH]$-endomorphisms of $\mathbb{V}$, namely $\bfM_4(
\bbF_d[\fH])$. Of course, a similar argument holds for any group $\fG$ and any subgroup $\fH$
of finite index, normal or not.

In our situation, $\fH$ is the free abelian group on $x=a^2, y=b^2$ and $z=c^2$.
Thus $\bbF_d[\fH] = \mathcal{L}_d(x,y,z)$, the Laurent polynomial ring in variables $x,y,z$ over $\bbF_d$,
and thus $\bbF_d[\fG]$ embeds in $\bfM_4(\mathcal{L}_d(x,y,x))$. Using capital letters for the matrices corresponding
to the generators of $\fG$, we have

\[ A=\left[ \begin{matrix} 0&1&0&0\\x&0&0&0\\0&0&0&x\inv y z\inv\\0&0& y\inv z&0 \end{matrix}\right] \qquad
B=\left[ \begin{matrix} 0&0&1&0\\0&0&0&1\\ y&0&0&0\\ 0&y\inv &0&0 \end{matrix} \right]
\]
\[ C=AB = \left[\begin{matrix} 0&0&0&1\\0&0&x&0\\ 0&x\inv z\inv&0&0\\z&0&0&0 \end{matrix} \right] \]
\[ X=A^2=\diag(x,x,x\inv,x\inv) \qquad Y=B^2 =\diag(y,y\inv,y,y\inv)\]
\[ Z= C^2 =\diag(z,z\inv,z\inv,z) \]

Following \cite{M} we choose a prime characteristic $d$ and two integer parameters $t$ and $w$
and, using $I$ for the identity matrix, we define the diagonal matrices
\[ H = (I - Z^{1-2t})^{d-2}\]
\[ P= (I+X)(I+Y)(Z^t+Z^{1-t}) H\]
\[ Q=Z^w[(I+X)(X\inv+Y\inv) +(I+Y\inv)(I+Z^{2t-1})]H \]
\[ R= Z^w[(I+Y\inv)(X+Y)Z^t +(I+X)(Z^t+Z^{1-t})]H \]
\[ S= Z^{2t-1} + (4I +X+X\inv +Y+Y\inv) H\]
These, of course, naturally correspond to elements of $\bbF_d[\fH]$ and thus
\[ U= P+QA+RB+SAB= P+QA+RB+SC\]
corresponds to an element of $\bbF_d[\fG]$. Indeed, it is shown in \cite{M} that
$U$ corresponds to a nontrivial unit of the group ring with inverse corresponding to a specific matrix of the
form
\[ U'=P'+Q'A+R'B+S'C\]
Note that, if $d=2$ then $H=I$, but $H$ is a nontrivial polynomial in $Z$ for $d>2$.

Now using any computer algebra system, it is an easy task to describe $U$ and $U'$ for any set of
parameters. However even for relatively small $d$, these matrices look unbelievably complicated and
fill numerous computer screens. But when we multiply $UU'$ and $U'U$ it is satisfying that we obtain
the $4\times 4$ identity matrix $I$. These computations verify the assertions in \cite{G} and \cite{M}
at least for the specific set of parameters. What is surprising in these computations,
for all the parameters we could check, is that
the determinants of $U$ and $U'$ always seem to be equal to 1. It is not clear why this should be,
but it is easily provable and we do so below. Note that $\det A=1$ and $\det B=1$ so all elements
of $\fG$ have matrices of determinant 1.

We remark that for general finite $\fG$ and $\fH$,
it is known that the determinant of this matrix representation is related to the group theoretic transfer map.

\begin{proposition}
For all parameters $d,t,w$, we have $\det U=\det U'=1$.
\end{proposition}

\begin{proof}
Fix a set of parameters. We ignore the group ring, but rather we work in the $4\times 4$ matrix ring over
the Laurent polynomial ring in $x,y,z$. Now $UU'=I$, so $\det U$ is a unit in $\mathcal{L}_d(x,y,z)$
with inverse $\det U'$. Thus $\det U= fx^iy^jz^k$ for some $0\neq f\in\bbF_d$ and integers $i,j,k$.

Consider the homomorphism $\cmpl{\phantom{x}}\colon \mathcal{L}_d(x,y,z)\to \mathcal{L}_d(x,y)$
given by $x\mapsto x$, $y\mapsto y$ and $z\mapsto 1$, and extend this to the corresponding
$4\times 4$ matrix rings. Then $\cmpl Z = I$. If $d>2$, then $\cmpl H=0$ so $\cmpl P = \cmpl Q =\cmpl R=0$
and $\cmpl S = I$. Thus $\cmpl U = \cmpl A \cmpl B$ has determinant 1. On the other hand, if $d=2$
then $\cmpl H = I$ and $\cmpl Z = I$, so $\cmpl U$ is independent of the parameters $w$ and $t$
and it is easy to check 
(using a computer algebra system, if necessary)
that $\det \cmpl U =1$ in this case also. Since $\det \cmpl U$ is the image of
$f x^iy^jz^k$ under $\cmpl{\phantom{x}}$, we see that $f x^iy^j=1$, so $f=1$ and $i=j=0$.
in other words, $\det U = z^k$.

Now consider the homomorphism $\tilde{\phantom{x}}\colon \mathcal{L}_d(x,y,z)\to \mathcal{L}_d(z)$
given by $x\mapsto -1$, $y\mapsto -1$ and $z\mapsto z$, and extend this to the corresponding
$4\times 4$ matrix rings. Then $\wt{X}=\wt{X}\inv=\wt{Y} =\wt{Y}\inv=-I$, so $P$, $Q$ and $R$ map
to 0 and $\wt{S}=\wt{Z}^{2t-1}$. It follows that $\wt{U} =\wt{Z}^{2t-1} \wt{C}$ has determinant 1.
But $\det\wt{U}$ is the image of $z^k$ under $\tilde{\phantom{x}}$, so we see that $z^k=1$ and $k=0$,
as required.
\end{proof}

Now let us return to the group algebra and let $u_0$ be the Murray nontrivial unit given by $t=0$ and $w=0$.
Then $u_0=p_0+ q_0 a +r_0 b +s_0 ab$ where $h_0=(1-z)^{d-2}$
\[p_0=(1+x)(1+y)(1+z)h_0\]
\[q_0=[(1+x)(x\inv+y\inv) +(1+y\inv)(1+z\inv)]h_0\]
\[r_0 = [(1+y\inv)(x+y)+(1+x)(1+z)] h_0\]
\[ s_0= z\inv +(4+x+x\inv +y+y\inv) h_0\]
The next question to ask is why is there a doubly infinite family of such nontrivial units. The answer
here is fairly easy, namely these units correspond to a doubly infinite family of endomorphisms
of $\fG$. Specifically

\begin{proposition}
Fix the characteristic $d$ and let $u$ be the nontrivial unit of $\bbF_d[\fG]$ corresponding to the
parameters $t$ and $w$. Then $u= z^t \sigma(u_0)$ where $\sigma$ is the endomorphism of
$\bbF_d[\fG]$ determined by the group endomorphism $\sigma\colon\fG\to\fG$ given by
$a\mapsto z^w z^{-t} a$ and $b\mapsto z^w b$.
\end{proposition}

\begin{proof}
Let $t$ and $w$ be integers and define group elements 
in $\fG$ by $\cmpl a =z^w z^{-t} a$ and $\cmpl b = z^w b$. Since $a$ and $b$ invert $z$
by conjugation, it follows that $\cmpl a^2 = a^2=x$ and $\cmpl b^2 = b^2=y$. Furthermore,
$\cmpl c = \cmpl a \cmpl b = (z^w z^{-t} a) (z^w b) = z^{-t} c$ so $\cmpl c^2 = z^{1-2t} =\cmpl z$. Now
$(\cmpl a^2) ^{\cmpl b} = (\cmpl a)^{-2}$ and $(\cmpl b^2)^{\cmpl a} =(\cmpl b)^{-2}$, so it follows
from the definition of $\fG$ that there exists a homomorphism $\sigma\colon \fG \to \fG$
with $\sigma(a)=\cmpl a$ and $\sigma(b)=\cmpl b$. Note that $\sigma(x) =\sigma(a)^2=x$,
$\sigma(y)=\sigma(b)^2=y$, $\sigma(c) =\cmpl c= z^{-t} c$ and $\sigma(z)=\cmpl c^2= z^{1-2t}$.
It follows that $\sigma\colon \fH\to \fH$ is one-to-one, but not necessarily onto, and then
$\sigma\colon \fG\to \fG$ is also one-to-one, but not necessarily onto.
Of course $\sigma$ extends to an algebra homomorphism $\sigma\colon \bbF_d[\fG] \to
\bbF_d[\fG]$. In particular, $\sigma$ sends units to units and, since $\sigma$ is one-to-one on $\fG$,
it sends nontrivial units to nontrivial units.

Now let $u=p+qa+rb+sc$ be the unit associated with $t$ and $w$. We compute $z^t \sigma(u_0)$
as follows. First $h_0=(1-z)^{d-2}$ so $\sigma(h_0) = (1-z^{1-2t})^{d-2} =h$. Next
$p_0=(1+x)(1+y)(1+z)h_0$ so 
\[z^t\sigma(p_0) =z^t (1+x)(1+y)(1+z^{1-2t}) h = p\]  and
$q_0= [(1+x)(x\inv+y\inv) +(1+y\inv)(1+z\inv)] h_0$ so
\[ z^t\sigma(q_0 a) = z^t z^wz^{-t}[(1+x)(x\inv+y\inv)+(1+y\inv)(1+z^{2t-1})] ha =q a\]
Similarly $r_0 = [(1+y\inv)(x+y)+(1+x)(1+z)] h_0$ so
\[z^{t}\sigma(r_0 b)=z^t z^w [(1+y\inv)(x+y) +(1+x)(1+z^{1-2t})] hb=rb\]
and $s_0= z\inv +(4+x+x\inv +y+y\inv) h_0$ so
\[ z^t \sigma(s_0c)=z^t z^{-t} [z^{2t-1}+(4+x+x\inv +y+y\inv) h] c= sc\]
We conclude that $z^t \sigma(u_0)=u$ and the proposition is proved.
\end{proof}

Of course, $\fG$ admits automorphisms that permute the cosets $\fH a$, $\fH b$ and $\fH c$
transitively. Furthermore, there are additional endomorphisms that fix the cosets of $\fH$, and
all of these yield additional nontrivial units in $\bbF_d[\fG]$. Finally a close look at the last paragraph of the
proof of \cite[Theorem 3]{M} shows that if we define
\begin{align*} h_0(z) &= \frac{1-z^{nd}}{(1-z)^2} =\frac{(1-z^d)}{(1-z)^2} (1+z^d+z^{2d} +\cdots+ z^{(n-1)d} )\\
&=(1-z)^{d-2}(1+z^d+z^{2d} +\cdots+ z^{(n-1)d} )
\end{align*}
for any integer $n\geq 1$,
then the expression for $h$ in Theorem 3 can be replaced by
\[  h = h_0(z^{1-2t})\in \bbF_d[\fG]\]
In this way, for any fixed characteristic $d>0$, by varying $n$, $t$, and $w$, we obtain a triply infinite family of counterexamples.
 Note that 
Propositions 1 and 2 above apply equally well to this more general situation although a small amount of additional work is needed
when $d=2$ in Proposition 1. Namely here we must observe that $\cmpl H$ could equal either $0$ or $I$ depending on whether
$n$ is even or odd.

\end{document}